\documentclass[12pt]{amsproc}
\usepackage{hyperref}

\newtheorem{theorem}{\sc Theorem}[section]
\newtheorem{lemma}[theorem]{\sc Lemma}
\newtheorem{proposition}[theorem]{\sc Proposition}

\begin{document}
\title[Commutators]
{On profinite groups in which commutators are covered by finitely many subgroups} 
\author[C. Acciarri]{Cristina Acciarri}

\address{%
Department of Mathematics\\ University of Brasilia\\
    Brasilia - DF, Brazil 
}

\email{acciarricristina@yahoo.it}

\author{Pavel Shumyatsky}
\address{ Department of Mathematics\\ University of Brasilia\\
    Brasilia - DF, Brazil }
\email{pavel@mat.unb.br}

\thanks{Supported by CAPES and CNPq-Brazil}


\begin{abstract}
For a family of group words $w$ we show that if $G$ is a profinite group in which all $w$-values are contained in a union of finitely many subgroups with a prescribed property, then $w(G)$ has the same property as well. In particular, we show this in the case where the subgroups are periodic or of finite rank. If $G$ contains finitely many subgroups $G_1,G_2,\dots,G_s$ of finite exponent $e$ whose union contains all $\gamma_k$-values in $G$, it is shown that $\gamma_k(G)$ has finite $(e,k,s)$-bounded exponent. If $G$ contains finitely many subgroups $G_1,G_2,\dots,G_s$ of finite rank $r$ whose union contains all $\gamma_k$-values, it is shown that $\gamma_k(G)$ has finite $(k,r,s)$-bounded rank.

\end{abstract}
\maketitle

\section{Introduction}

A  covering of a group $G$ is a family $\{S_i\}_{i\in I}$ of subsets of
$G$ such that $G=\bigcup_{i\in I}\,S_i$.
If $\{H_i\}_{i\in I}$ is a covering of $G$ by subgroups, it is natural to ask
what information about $G$ can be deduced from properties of the subgroups $H_i$. In the case where the covering is finite actually quite a lot about the structure of $G$ can be said. The first result in this direction is due to Baer (see \cite{Neu1}), who proved that $G$ admits a finite covering by abelian subgroups if and only if it is central-by-finite. The nontrivial part of this assertion is an immediate consequence of a subsequent result of B.H. Neumann \cite{Neu2}: if $\{S_i\}$ is a finite covering of $G$ by cosets of subgroups, then $G$ is also covered by the cosets $S_i$ corresponding to subgroups of finite index in $G$. In other words, we can get rid of the cosets of subgroups of infinite index without losing the covering property.

Given a group word $w=w(x_1,\dots,x_n)$, we think of it primarily as a function of $n$ variables defined on any given group $G$. We denote by $w(G)$ the verbal subgroup of $G$ generated by the values of $w$. If the set of all $w$-values in a group $G$ can be covered by finitely many subgroups, one could hope to get some information about the structure of the verbal subgroup $w(G)$.

In this direction we mention the following result that was obtained in \cite{roga}. Let $w$ be either the lower central word $\gamma_k$ or the derived word $\delta_k$. Suppose that $G$ is a group in which  all  $w$-values are contained in a union of finitely many Chernikov subgroups. Then $w(G)$ is Chernikov.

Another result of this nature was established in \cite{wigustavo}: If $G$ is a group in which  all commutators are contained in a union of finitely many cyclic
subgroups, then $G'$ is either cyclic or finite.

In the present paper we deal with profinite groups in which all $w$-values are contained in a union of finitely many subgroups with certain prescribed properties.  A profinite group is a topological group that is isomorphic to an inverse limit of finite groups. The textbooks \cite{riza} and \cite{wilson} provide a good introduction to the theory of profinite groups. In the context of profinite groups all the usual concepts of group theory are interpreted topologically. In particular, by a subgroup of a profinite group we mean a closed subgroup. A subgroup is said to be generated by a set $S$ if it is topologically generated by $S$.

The words $w$ considered in this paper are the so-called outer (multilinear) commutator words. Recall that an outer commutator word is a word which is obtained by nesting commutators but using always different variables. A word of this kind has a form of a multilinear Lie monomial. For example the word $[[x_1,x_2],[y_1,y_2,y_5],z]$ is outer commutator while the Engel word $[x_1,x_2,x_2,x_2]$ is not. The word $w(x)=x$ in one variable is an outer commutator word of weight $1$; if $u$ and $v$  are defined outer commutator words of weights $m$ and $n$ respectively then, the word $w=[u,v]$ is an outer commutator word of weight $m+n$. An important family of outer commutator words are the lower central words $\gamma_k$,
given by
\[
\gamma_1=x_1,
\qquad
\gamma_k=[\gamma_{k-1},x_k]=[x_1,\ldots,x_k],
\quad
\text{for $k\ge 2$.}
\]
The corresponding verbal subgroups $\gamma_k(G)$ are the terms of the lower central series of $G$. Another distinguished sequence of outer commutator words are the derived words $\delta_k$, on $2^k$ variables, which are defined recursively by
\[
\delta_0=x_1,
\quad
\delta_k=[\delta_{k-1}(x_1,\ldots,x_{2^{k-1}}),\delta_{k-1}(x_{2^{k-1}+1},\ldots,x_{2^k})],
\quad
\text{for $k\ge 1$.}
\]
The verbal subgroup that corresponds to the word $\delta_k$ is the familiar $k$th derived subgroup of $G$ usually denoted by $G^{(k)}$.

In the next section we consider profinite groups $G$ having finitely many periodic subgroups $G_1,G_2,\dots,G_s$ whose union contains all $w$-values in $G$ for some outer commutator word $w$. Recall that a group is periodic (torsion) if every element of the group has finite order. A group is called locally finite if each of its finitely generated subgroups is finite. It is immediate from the definitions that every locally finite group is periodic. However the converse in general is not true -- there exist finitely generated infinite periodic groups (see for example Gupta's essay \cite{gupta}).  Periodic profinite groups have received a good deal of attention in the past. In particular, using Wilson's reduction theorem \cite{wil}, Zelmanov has been able to prove local finiteness of periodic compact groups \cite{z}. Earlier Herfort showed that there exist only finitely many primes dividing the orders of elements of a periodic profinite group \cite{he}. It is a long-standing problem whether any periodic profinite group has finite exponent. Recall that a group is said to be of exponent $e$ if $x^e=1$ for all $x\in G$ and $e$ is the least positive integer with that property.

The following theorem is the main result of the next section.

\begin{theorem}\label{222} Let $w$ be an outer commutator word and $G$ a profinite group that has finitely many periodic subgroups $G_1,G_2,\dots,G_s$ whose union contains all $w$-values in $G$. Then $w(G)$ is locally finite.
\end{theorem}

It follows from the proof that if under the hypothesis of the above theorem the subgroups $G_1,G_2,\dots,G_s$ have finite exponent, then $w(G)$ has finite exponent as well. In Section 3 we address the question whether the exponent of $w(G)$ is bounded in terms of the exponents of $G_1,G_2,\dots,G_s$ and $s$. Using the Lie-theoretic techniques that Zelmanov created in his solution of the restricted Burnside problem \cite{zel}, we obtained the following related result.

\begin{theorem}\label{333} Let $e,k,s$ be positive integers and $G$ a profinite group that has  subgroups $G_1,G_2,\dots,G_s$ whose union contains all $\gamma_k$-values in $G$. Suppose that each of the subgroups $G_1,G_2,\dots,G_s$ has finite exponent dividing $e$. Then $\gamma_k(G)$ has finite $(e,k,s)$-bounded exponent.
\end{theorem}

In Section 4 we study the case where all $w$-values are contained in a union of finitely many subgroups of finite rank. A group $G$ is said to be of finite rank $r$ if every finitely generated subgroup of $G$ can be generated by $r$ elements. The arguments used in the proofs of Theorems \ref{222} and \ref{333} turned out to be useful for the case of finite rank. Thus, we obtain the following theorems.

\begin{theorem}\label{444} Let $w$ be an outer commutator word and $G$ a profinite group that has finitely many subgroups $G_1,G_2,\dots,G_s$ whose union contains all $w$-values in $G$. If each of the subgroups  $G_{1},G_{2},\ldots,G_{s}$ is of finite rank, then $w(G)$ has finite rank as well.
\end{theorem}

\begin{theorem}\label{555} Let $k,r,s$ be positive integers and $G$ a profinite group that has  subgroups $G_1,G_2,\dots,G_s$ whose union contains all $\gamma_k$-values in $G$. Suppose that each of the subgroups $G_1,G_2,\dots,G_s$ has finite rank at most $r$. Then $\gamma_k(G)$ has finite $(k,r,s)$-bounded rank.
\end{theorem}

Unlike the situation of Theorem \ref{333}, the proof of Theorem \ref{555} does not use Zelmanov's Lie-theoretic techniques. Instead, an important r\^ole in the proof is played by the Lubotzky--Mann theory of powerful $p$-groups \cite{lumann}.

Throughout the paper the expression ``$(a, b, \ldots)$-bounded'' stands for ``bounded from above by a function depending only on the parameters $a$, $b$, $\ldots$''.

\section{Local finiteness of verbal subgroups} 
Our goal in this section is to prove Theorem \ref{222}. Zelmanov's theorem that every periodic compact group is locally finite will be used without explicit references. We start with some technical lemmas.   
\begin{lemma}
\label{111} 
Let $G$ be a group, $k$ a positive integer  and $h,a_1,\dots,a_{2^k}$ elements of $G$. Let  us denote by $x_j$ and $y_{j}$ the  two elements of $G$  obtained by replacing in $\delta_k(a_1,\dots,a_{2^k})$ the entry $a_j$ with $a_jh$ and  $h$, respectively. Then there exist $h_1,\dots,h_{2^k}\in \langle h^G\rangle$ such that $$x_j=\delta_k(a_1^{h_1},\dots,a_{2^k}^{h_{2^k}})y_j.$$
\end{lemma}

\begin{proof}  We argue by induction on $k$. Assume first that $k=1$. Denote $a_1^{-1}h{a_1}$ by $t$. Using the well-known commutator identities write
$$x_1=[a_1h,a_2]=[a_1^h,a_2^h][h,a_2]=[a_1^h,a_2^h]y_1$$ and 
$$x_2=[a_1,a_2h]=[a_1,h][a_1,a_2]^h=[a_1^t,a_2^t]y_2.$$ This shows that if $k=1$, the lemma is correct. Hence, we assume that $k\geq 2$ and use induction on $k$. Without loss of generality we also assume that $j\geq 2^{k-1}$. Let  us denote by $r_j$ the element of $G$ obtained by replacing in $\delta_{k-1}(a_{2^{k-1}+1},\dots,a_{2^k})$ the entry $a_j$ with $a_jh$ and  by $s_j$ the element of $G$ obtained by replacing in $\delta_{k-1}(a_{2^{k-1}+1},\dots,a_{2^k})$ the entry $a_j$ with $h$. By the inductive hypothesis  there exist $g_1,\dots,g_{2^{k-1}}\in \langle h^G\rangle$ such that 
\begin{equation}
\label{rj}
r_j=\delta_{k-1}(a_{2^{k-1}+1}^{g_1},\dots,a_{2^k}^{g_{2^{k-1}}})s_j.
\end{equation} 
For the sake of  simplicity we denote $\delta_{k-1}(a_1,\dots,a_{2^{k-1}})$ by $\Delta$ and observe that  $x_j=[\Delta,r_j] $ and $y_{j}=[\Delta,s_j]$ .   So by (\ref{rj}) we have 
\begin{equation*} 
\begin{split}
x_j=[\Delta,r_j] =&[\Delta,\delta_{k-1}(a_{2^{k-1}+1}^{g_1},\dots,a_{2^k}^{g_{2^{k-1}}})s_j]=\\
& [\Delta,\delta_{k-1}(a_{2^{k-1}+1}^{g_1},\dots,a_{2^k}^{g_{2^{k-1}}})]^{s_jy_j^{-1}}[\Delta,s_j].
\end{split}
\end{equation*}
Since it is clear that  both elements $s_j$ and $y_j$ belong to $\langle h^G\rangle$,  the lemma follows.
\end{proof}

Let $G$ be a group  and  $w=w(x_{1},\ldots,x_{n})$ a nontrivial  group word. Suppose that $G$ has  a subgroup $H$ and elements $g_{1},\ldots,g_{n}$ such that $w(g_{1}h_{1},\ldots,g_{n}h_{n})=1$ for all $h_{1},h_{2},\ldots, h_{n}\in H$. Then we say that the law $w\equiv 1$ is satisfied on the cosets $g_{1}H,g_{2}H, \ldots, g_{n}H$.   
\begin{lemma}
\label{22} Let $G$ be a group and $H$ a normal subgroup of $G$. Let $k$ be a positive integer and  suppose that $a_1,\dots,a_{2^k}$ are elements of $G$ such that  the law $\delta_{k}\equiv 1$ is satisfied on the cosets $a_1H,\dots,a_{2^k}H$. Then $H$ is soluble with derived length at most $k$.
\end{lemma}
\begin{proof} 
By the hypothesis all elements of the form $\delta_k(b_{1},\dots,b_{2^k})$, where each entry $b_{i}$ belongs to $a_{i}H$, are trivial. Let $y$ be an element that can be obtained from $\delta_k(b_{1},\dots,b_{2^k})$ by replacing some of the entries $b_1,\dots,b_{2^k}$ with elements of $H$. We wish to show that $y=1$. Suppose that $y$ is obtained from $\delta_k(b_{1},\dots,b_{2^k})$ by replacing precisely $n$ of the entries and argue by induction on $n$. If $n=0$, then $y=\delta_k(b_{1},\dots,b_{2^k})=1$. Thus, assume that $n\geq 1$ and use induction on $n$. Suppose now that $y=\delta_k(c_1,\dots,c_{2^k})$, where each $c_i$ equals either  $b_{i}\in a_iH$ or $h_i\in H$, and choose an index $j$ such that $c_j=h_j$. 

Let $x$ be the element obtained from $\delta_{k}(c_{1},\ldots,c_{j-1},a_{j},c_{j+1},\ldots,c_{2^{k}})$ by replacing $a_j$ with $a_jh_j$. Furthermore let us observe that $y$ equals the element obtained  from $\delta_{k}(c_{1},\ldots,c_{j-1},a_{j},c_{j+1},\ldots,c_{2^{k}})$ by replacing  the entry $a_j$ with $h_j$. Applying Lemma \ref{111} to  $\delta_{k}(c_{1},\ldots,a_{j},\ldots,c_{2^{k}})$ with  $h=h_{j}$ we conclude that there exist $g_1,\dots,g_{2^k}\in H$ such that 
\begin{equation}
\label{x}
x=\delta_k(c_1^{g_1},\dots,a_j^{g_j},\dots,c_{2^k}^{g_{2^k}})y.
\end{equation}
Since both $x$ and $\delta_k(c_1^{g_1},\dots,a_j^{g_j},\dots,c_{2^k}^{g_{2^k}})$ can be obtained from elements of type $\delta_k(b_{1},\dots,b_{2^k})$ by replacing $n-1$ entries with elements of $H$, by induction we have $x=1$ and $\delta_k(c_1^{g_1},\dots,a_j^{g_j},\dots,c_{2^k}^{g_{2^k}})=1$.  Thus  by (\ref{x}) also $y$ must be $1$. 

In the particular case where $n=2^k$ we have $\delta_k(h_1,\dots,h_{2^k})=1$ for all $h_1,\dots,h_{2^k}\in H$. It follows that $H$ is soluble with derived length at most $k$, as desired.
\end{proof}

In the next lemma we shall require the concept of  the marginal subgroup corresponding to a word.  Let $G$ be a group and $w=w(x_{1},\ldots,x_{n})$ any word. The marginal subgroup $w^{*}(G)$ of $G$ corresponding to the word $w$ is defined as the set of all $a\in G$ such that  
$$w(g_{1},\ldots,ag_{i},\ldots,g_{n})=w(g_{1},\ldots,g_{i},\ldots,g_{n}),$$ for all $g_{1},\ldots,g_{n} \in G$ and $ 1\leq i\leq n$. It is well known that $w^{*}(G)$ is a characteristic subgroup of $G$ and that $[w^{*}(G),w(G)]=1.$
If $w$ is an outer commutator word, then $w^{*}(G)$ is precisely the set $S$ such that  $w(g_{1},\ldots, g_{n})=1$ whenever at least one of the elements $g_{1},\ldots,g_{n}$ belongs to $S$. A proof of this can be found in \cite[Theorem 2.3]{turner}. The next lemma was obtained in discussions of the first author with G.A. Fern\'andez-Alcober.   

\begin{lemma}\label{lemma nontrivial w-values} Let $G$ be a group and $w$ any outer commutator word.  If $N$ is a normal subgroup of $G$ containing no nontrivial $w$-values, then $[N,w(G)]=1$.
\end{lemma}
\begin{proof} 
Let $k$ be the weight of $w$. Since $N$ is normal and $w$ is an outer commutator word, it follows that whenever $x\in N$ every element of the form  $$w(g_{1},\ldots, g_{i-1},x,g_{i+1},\ldots,g_{k})$$ belongs to $N$. Thus, by the hypothesis it must be trivial. Therefore $N\leq w^{*}(G)$. The result is now clear since $w^{*}(G)$ always commutes with $w(G)$. 
\end{proof}

A proof of the next lemma can be found in \cite[Lemma 4.1]{Shu}.
\begin{lemma}\label{um} Let $G$ be a group and $w$ an outer commutator word of weight $k$. Then every $\delta_k$-value in $G$ is  a $w$-value.
\end{lemma}
\begin{lemma}\label{schur} Let $G$ be a soluble-by-finite profinite group and suppose that $G/Z(G)$ is periodic. Then $G'$ has finite exponent. In particular $G'$ is locally finite.
\end{lemma}
\begin{proof} It is well-known that a soluble periodic profinite group has finite exponent. For abelian groups this is Exercise 10(c) in \cite[page 45]{wilson} and in the general case the result follows easily by induction on the derived length. Thus, the exponent of $G/Z(G)$ is finite. We denote this by $e$. A theorem of Mann says that if $K$ is a finite group such that $K/Z(K)$ is of exponent $e$, then the exponent of $K'$ is $e$-bounded \cite{mann}. Applying now a profinite version of this theorem we conclude that the exponent of $G'$ is finite.
\end{proof}

\begin{proposition}\label{soluble_local finit} Let $G$ be a profinite group and $w$ any outer commutator word. Suppose that $G$ has finitely many periodic subgroups whose union contains all $w$-values in $G$. If $G$ contains an open soluble subgroup, then $w(G)$ is locally finite.
\end{proposition}
\begin{proof} Let $d$ be the minimal among derived lengths of the open normal soluble subgroups of $G$. If $d=0$, then $G$ is finite and there is nothing to prove. So we assume that $d\geq 1$ and use induction on $d$. Thus, $G$ is infinite and let us choose an open normal soluble subgroup $K$ in $G$ such that the derived length of $K$ is precisely $d$.  Let $N$ be the last nontrivial term of the derived series of $K$. Denote by $M$ the subgroup generated by all $w$-values that belong to $N$ and let $G_1,G_2,\dots,G_s$ be the finitely many periodic subgroups whose union contains all $w$-values in $G$. We have $M=M_1\cdots M_s$, where $M_i=M\cap G_i$, for $i=1,\ldots,s$. Each subgroup $M_{i}$ is locally finite,  since it is a periodic profinite group,  and so it follows that also $M$ is locally finite.  We can pass to the quotient   $G/M$ and assume that $N$ contains no  nontrivial $w$-values. By Lemma \ref{lemma nontrivial w-values}  we have $w(G)\cap N\leq Z(w(G))$.  By induction on $d$ we can assume that the image of $w(G)$ in $G/N$ is locally finite and so $w(G)/Z(w(G))$ is locally finite. Lemma \ref{schur} now shows that the derived subgroup of $w(G)$ is locally finite. Thus, passing to the quotient $G/w(G)'$ we may assume that $w(G)$ is abelian. Then of course $w(G)$ is the product of subgroups $w(G)\cap G_{i}$ for $i=1,\ldots,s$. Since $w(G)$ is an abelian group which is a product of  finitely many periodic subgroups, it follows that $w(G)$ is locally finite. The proof is complete.
\end{proof}  

We are now in the position to prove Theorem \ref{222}.

\begin{proof}[Proof of Theorem \ref{222}] Let $k$ be the weight of $w$. According to Lemma \ref{um} every $\delta_k$-value is also a $w$-value. For each integer $i=1,\dots,s$ we set 
\begin{equation*}
S_i=\Big\{(x_1,\dots,x_{2^k})\in \underbrace{G\times\dots\times G}_{2^k}\mid \delta_k(x_1,\dots,x_{2^k})\in G_i\Big\}.
\end{equation*} 
Note that the sets $S_i$ are closed in $\underbrace{G\times\dots\times G}_{2^{k}}$ and cover  the whole of $\underbrace{G\times\dots\times G}_{2^{k}}$. By Baire's Category Theorem \cite[p.\ 200]{Kell} at least one of these sets contains a non-empty interior. Hence, there exist an open subgroup $H$ of $G$, elements $a_1,\dots,a_{2^k}$ in $G$ and an integer $j$ such that 
\begin{equation}
\delta_k(a_1h_1,\dots,a_{2^k}h_{2^k})\in G_j  \text{ for all } h_1,\dots,h_{2^k}\in H.
\end{equation}
Without loss of generality we can assume that the subgroup $H$ is normal. In this case $H$ normalizes the set of all commutators of the form $\delta_k(a_1h_1,\dots,a_{2^k}h_{2^k})$, where $h_1,\dots,h_{2^k}\in H$. Let $K$ be the subgroup of $G$ generated by all commutators of the form $\delta_k(a_1h_1,\dots,a_{2^k}h_{2^k})$, where $h_1,\dots,h_{2^k}\in H$. Note that  $K\leq G_{j}$. Since the subgroup $G_{j}$ is  locally finite, so is  $K$.  Let $D=K\cap H$. Then $D$ is a normal locally finite subgroup of $H$ and the normalizer  of $D$ in $G$ has finite index. Therefore there are only finitely many conjugates of $D$ in $G$. Let $D=D_1,D_2, \dots,D_r$ be all these conjugates. All of them are normal in $H$ and so their product $D_{1}D_{2}\cdots D_{r}$ is locally finite. By passing to the quotient $G/D_1D_2\cdots D_r$  we may assume that $D=1$.  Since $D=K\cap H$ and $H$ has finite index in $G$, it follows that $K$ is finite.  On the other hand,  the normalizer of $K$ has finite index in $G$ and so the normal closure, say $L$, of $K$ in $G$ is also finite. We can pass to the quotient group $G/L$ and assume that $K=1$. In that case we have $\delta_k(a_1h_1,\dots,a_{2^k}h_{2^k})=1$ for all $h_1,\dots,h_{2^k}\in H$.  Now by Lemma \ref{22}  the subgroup $H$ is soluble and so the theorem follows immediately from  Proposition \ref{soluble_local finit}.
\end{proof}

\section{Bounding the exponent of $\gamma_{k}(G)$}
\label{exponent}
The aim of this section is to prove Theorem \ref{333}. In the next lemma we call a subset $X$ of a group commutator-closed to mean that $[x,y]\in X$ whenever $x,y\in X$. It is clear that the set of all $\gamma_k$-values is commutator-closed in any group.
\begin{lemma}\label{751} Let $G$ be a nilpotent group generated by a commutator-closed subset $X$ which is contained in a union of finitely many subgroups $G_1,G_2,\dots,G_s$. Then $G$ can be written as the product $G=G_1G_2\cdots G_s$.
\end{lemma}
\begin{proof} We use induction on the nilpotency class $c$ of $G$. If $G$ is abelian, the lemma is clear so we assume that $c$ is at least 2. Let $K=\gamma_c(G)$ and set $K_i=K\cap G_i$ for $i=1,2,\dots,s$. Then $K$ is central in $G$ and is generated by commutators in elements of $X$. It follows that $K=K_1K_2\cdots K_s$. By induction we assume that $$G=G_1G_2\cdots G_sK=G_1G_2\cdots G_sK_1K_2\cdots K_s.$$ Since the subgroups $K_1,K_2,\dots,K_s$ are central and $K_i\leq G_i$, we deduce that $G=G_1G_2\cdots G_s$, as required.
\end{proof}

Let $P$ be a Sylow $p$-subgroup of a finite group $G$. An immediate corollary of the Focal Subgroup Theorem \cite{go} is that $P\cap G'$ is generated by commutators. We do not know if $P\cap w(G)$ is generated by $w$-values for every outer commutator word $w$. A proof of the next related result can be found in \cite[Theorem A]{AFS}.  
\begin{theorem}\label{752} Let G be a finite group, $p$ a prime, and  $P$ a Sylow $p$-subgroup of $G$. If $w$ is an outer commutator word, then $P\cap w(G)$ is generated by  powers of $w$-values.
\end{theorem}

In this section the above theorem is applied in the case where $w$ is a lower central word. 

The routine inverse limit argument shows that Theorem \ref{333} can be easily deduced from the corresponding result for finite groups. Thus, we will deal with the case where $G$ is finite. In the proof of the next theorem we use a result from \cite{novar} whose proof is based on Zelmanov's Lie-theoretic techniques.

\begin{theorem}
\label{exponent gammak finite} Let $e$, $k$, $s$ be positive integers. Let $G$ be a finite group that has subgroups $G_1,G_2,\dots,G_s$ whose union contains all $\gamma_{k}$-values of $G$. Suppose that each subgroup $G_i$ has finite exponent dividing $e$. Then $\gamma_{k}(G)$ has finite exponent bounded by a function depending on $e$, $k$ and $s$ only.
\end{theorem}

\begin{proof} Let $P$ be a Sylow $p$-subgroup of $\gamma_k(G)$. It is sufficient to show that the exponent of $P$ is $(e,k,s)$-bounded. Theorem \ref{752} shows that $P$ is generated by elements of order dividing $e$. Therefore it is sufficient to show that the exponent of $\gamma_k(P)$ is $(e,k,s)$-bounded. Thus, we may assume from the beginning that $G$ is a $p$-group. Further, without loss of generality we may assume that each subgroup $G_{i}$ is generated by $\gamma_{k}$-values and so $\gamma_{k}(G)=\langle G_{1}, G_{2}\ldots, G_{s} \rangle$. It follows from Lemma \ref{751} that $\gamma_{k}(G)=G_{1}G_{2}\cdots G_{s}$. We wish to prove that an arbitrary element of $\gamma_k(G)$ has bounded order. Thus, choose $g\in\gamma_{k}(G)$. We know that there exist elements $g_i\in G_i$ for $i=1,2,\dots,s$ such that $g=g_1g_2\cdots g_s$. Set $R=\langle g_1,g_2,\dots,g_s\rangle$. The subgroup $R$ is generated by $s$ elements, each of order dividing $e$ and every $\gamma_k$-value in $R$ is of order dividing $e$ as well. By \cite[Theorem 7]{novar} the order of $R$ is $(e,k,s)$-bounded. In particular, the order of $g$ is $(e,k,s)$-bounded, as required.
\end{proof}
The proof of Theorem \ref{333} is now complete.

\section{ Bounding the rank of a verbal subgroup}
\label{rank}

In this section we will prove  Theorems \ref{444} and \ref{555}. First we will show that Theorem \ref{444} holds under the additional assumption that $G$ contains an open soluble subgroup. 
\begin{proposition}
\label{soluble_rank} Let $G$ be a profinite group and $w$ any outer commutator word. Suppose that $G$ has finitely many subgroups of finite rank whose union contains all $w$-values in $G$.  If $G$ contains an open soluble subgroup, then the rank of $w(G)$ is finite.
\end{proposition}
\begin{proof} Arguing as in the proof of Proposition \ref{soluble_local finit} let us define $d$ to be the minimum of the derived lengths of the open normal soluble subgroups of $G$. If $d=0$, then $G$ is finite and there is nothing to prove. So we use induction on $d$ and assume that $d\geq 1$. In particular $G$ is infinite and choose an open normal soluble subgroup $H$ of $G$ such that the derived length of $H$ is precisely $d$.  Let $N$ be the last nontrivial term of the derived series of $H$ and let $M$ be the subgroup generated by all $w$-values that belong to $N$. Let $G_1,G_2,\dots,G_s$ be the finitely many subgroups of finite rank whose union contains all $w$-values in $G$. We have $M=M_1\cdots M_s$, where $M_i=M\cap G_i$ for $i=1,\ldots,s$.  Note that each subgroup $M_{i}$ is of finite rank, since $G_{i}$ has finite rank by the hypothesis. It follows that $M$ has finite rank as well. We pass to the quotient $G/M$ and assume that $N$ contains no nontrivial $w$-values. By Lemma \ref{lemma nontrivial w-values}  we have $w(G)\cap N\leq Z(w(G))$.  By induction on $d$ we  assume that the image of $w(G)$ in $G/N$ is of finite rank and so $w(G)/Z(w(G))$ is of finite rank. Theorem 2.5 in \cite{kurda1995} tells us that if $K$ is a soluble-–by-–finite group such that  $K/Z(K)$  has finite rank, then $K'$ has finite rank as well. Moreover the rank of $K'$ is bounded in terms of the derived length of the soluble radical of $K$, its index in $K$ and the rank of $K/Z(K)$. The profinite version of this result is straightforward and so we are in a position to apply it with $K=w(G)$. It follows that the rank of the derived group of $w(G)$ is finite. Thus, passing to the quotient $G/w(G)'$ we may assume that $w(G)$ is abelian. Then $w(G)$ is a product of  finitely many  subgroups of finite rank and we conclude that the rank of $w(G)$ is finite, as required.
\end{proof}  

We can now deal with the general case.  
\begin{proof}[Proof of Theorem \ref{444}]
Let $k$ be the weight of $w$. According to Lemma \ref{um} every $\delta_k$-value is also a $w$-value. For each integer $i=1,\dots,s$ we set 
\begin{equation*}
S_i=\Big\{(x_1,\dots,x_{2^k})\in \underbrace{G\times\dots\times G}_{2^k}\mid \delta_k(x_1,\dots,x_{2^k})\in G_i\Big\}.
\end{equation*} 
The sets $S_i$ are closed in $\underbrace{G\times\dots\times G}_{2^{k}}$ and cover  the group $\underbrace{G\times\dots\times G}_{2^{k}}$. By Baire's Category Theorem \cite[p.\ 200]{Kell} at least one of these sets contains a non-empty interior. Hence, there exist an open subgroup $H$ of $G$, elements $a_1,\dots,a_{2^k}$ in $G$ and an integer $j$ such that 
$$
\delta_k(a_1h_1,\dots,a_{2^k}h_{2^k})\in G_j  \text{ for all } h_1,\dots,h_{2^k}\in H.
$$
Without loss of generality we can assume that the subgroup $H$ is normal. Then it is clear that $H$ normalizes the set of all commutators of the form $\delta_k(a_1h_1,\dots,a_{2^k}h_{2^k})$, where $h_1,\dots,h_{2^k}\in H$. Let $K$ be the subgroup of $G$ generated by all commutators of the form $\delta_k(a_1h_1,\dots,a_{2^k}h_{2^k})$, where $h_1,\dots,h_{2^k}\in H$. Note that  $K\leq G_{j}$. Since the  rank of  $G_{j}$ is finite by the hypothesis, it follows that  also  $K$ has finite rank.  Let $D=K\cap H$. Then $D$ is a normal subgroup of finite rank in $H$ and the normalizer  of $D$ in $G$ has finite index. Therefore there are only finitely many conjugates of $D$ in $G$. Let $D=D_1,D_2,\dots,D_r$ be all these conjugates. All of them are normal in $H$ and so their product  $D_{1}D_{2}\cdots D_{r}$ has finite rank. By passing to the quotient $G/D_1D_2\cdots D_r$  we may assume that $D=1$.  Since $D=K\cap H$ and $H$ has finite index in $G$, it follows that $K$ is finite.  On the other hand,  the normalizer of $K$ has finite index in $G$ and so the normal closure, say $L$, of $K$ in $G$ is also finite. Thus we can pass to the quotient $G/L$ and assume that $K=1$. In this case we have $\delta_k(a_1h_1,\dots,a_{2^k}h_{2^k})=1$ for all $h_1,\dots,h_{2^k}\in H$. By Lemma \ref{22} the subgroup $H$ is soluble and so the theorem follows immediately from  Proposition \ref{soluble_rank}.
\end{proof}

We will now proceed to the proof of Theorem \ref{555}. 
The result can be easily deduced from the corresponding result on finite groups. Thus we will deal with the case where $G$ is a finite group. We will mention now some results about rank of a finite group. The following theorem is an immediate consequence of a result obtained independently by Guralnick \cite{Gu} and Lucchini \cite{Lu}.  It depends on the classification of finite simple groups. 
The result for soluble groups was obtained by Kov\'acs \cite{Ko}.

\begin{theorem}
\label{bound for rank}
Let $G$ be a finite group in which the rank of every Sylow subgroup is at most $r$. Then the rank of $G$ is at most $r+1$.
\end{theorem}

The next lemma essentially is due to Lubotzky and Mann \cite{lumann}. For the reader's convenience we give a proof here.
\begin{lemma}
\label{lemma product}
Let $P$ be a finite $p$-group and suppose that every term of the lower central series of $P$ can be generated by $d$ elements. Then the rank of $P$ is $d$-bounded.
\end{lemma}

\begin{proof} Let $V$ be the intersection of kernels of all homomorphisms of $P$ into $GL_{d}(\mathbb{F})$, where $\mathbb{F}$ is the field with $p$ elements. Set $W=V$ if $p$ is odd and $W=V^{2}$ if $p=2$. Then any characteristic $d$-generated subgroup of $P$ contained in $W$ is powerful by Proposition 2.12 in \cite{dsms}. Since the Sylow $p$-subgroups of  $GL_{d}(\mathbb{F})$ are nilpotent of class $d-1$, it follows that $\gamma_{d}(P) \leq V$. We know that $\gamma_{d}(P)$ is $d$-generated so the image of $\gamma_{d}(G)$ in $P/W$ has order at most $2^{d}$. Therefore $P/W$ is nilpotent of class at most $2d-1$ whence $\gamma_{2d}(P)\leq W$. Since $\gamma_{2d}(P)$ has at most $d$ generators, it becomes clear that $\gamma_{2d}(P)$ is powerful. Thus we conclude that $\gamma_{2d}(P)$ has rank at most $d$ \cite[Theorem 2.9]{dsms}. Since $P$ has at most $d$ generators, it is easy to see that the rank of $P/\gamma_{2d}(P)$ is $d$-bounded. The lemma follows.
\end{proof}

\begin{theorem}
\label{bounded rank finite} 
Let $k,r$ and $s$  be positive integers. Let $G$ be a finite group that has  subgroups $G_1,G_2,\dots,G_s$ whose union contains all $\gamma_{k}$-values of $G$. Suppose that each subgroup $G_i$ is of rank at most $r$. Then  the rank of $\gamma_{k}(G)$ is  bounded by a function depending on $k,r$ and $s$ only.
\end{theorem}
\begin{proof} We remark that our hypothesis implies that if $H$ is a subgroup of $G$ generated by the intersections $H\cap G_i$, then $H$ can be generated by at most $rs$ elements.
Hence, for every $j\geq k$ and every subgroup $U$ of $G$ the corresponding term of the lower central series $\gamma_j(U)$ can be generated by at most $rs$ elements. It follows from Lemma \ref{lemma product} that every $p$-subgroup with $rs$ generators has $\{k,r,s\}$-bounded rank. Now let $P$ be a Sylow $p$-subgroup of $\gamma_k(G)$. Theorem \ref{752} tells us that  $P$ is generated by powers of $\gamma_{k}$-values. Obviously the powers are contained in the union of the subgroups $G_{i}$. Hence $P$ can be generated by at most $rs$ elements and so the rank of $P$ is $\{k,r,s\}$-bounded. Now the theorem is immediate from Theorem \ref{bound for rank}.
\end{proof}
The proof of Theorem \ref{555} is now complete.


\end{document}